\documentclass{article}

\usepackage[english]{babel}
\usepackage{amssymb}
\usepackage{amsthm}
\usepackage{amsmath}
\usepackage{a4wide}
\usepackage{esvect}
\usepackage{hyperref}
\hypersetup{
    colorlinks,
    linkcolor=red,
    citecolor=black,
    urlcolor=blue
}
\usepackage{verbatim}
\usepackage[dvips]{graphicx}
\usepackage{t1enc}
\usepackage{color}
\usepackage{geometry}

\usepackage{tikz}
\usepackage{caption}
\usepackage{subcaption}
\newtheorem{thm}{Theorem}
\newtheorem{lem}[thm]{Lemma}

\newtheorem{conjecture}[thm]{Conjecture}
\newtheorem{proposition}[thm]{Proposition}

\newtheorem{corollary}[thm]{Corollary}

\newcommand{\ignore}[1]{}

\title{On a conjecture of Gentner and Rautenbach} 

\author{Ant\'onio Gir\~ao\thanks{Department of Pure Mathematics and
    Mathematical Statistics, University of Cambridge, Cambridge, UK; \texttt{A.Girao@dpmms.cam.ac.uk}} \and G\'abor M\'esz\'aros \thanks{Department of Mathematical Sciences, The University of Memphis, Memphis, TN; \texttt{gmszaros@memphis.edu}} \and Stephen G. Z. Smith \thanks{Department of Mathematical Sciences, The University of Memphis, Memphis, TN; \texttt{sgsmith1@memphis.edu}}}

\begin{document}

\maketitle

\begin{abstract}
Gentner and Rautenbach conjectured that the size of a minimum zero forcing set in a connected graph on $n$ vertices with maximum degree $3$ is at most $\frac{1}{3}n+2$. We disprove this conjecture by constructing a collection of connected graphs $\{G_n\}$ with maximum degree 3 of arbitrarily large order having zero forcing number at least $\frac{4}{9}|V(G_n)|$. 
\end{abstract}
\maketitle 

\section{Introduction}\label{sec:intro}
The Zero Forcing Number of a graph was first introduced by Burgarth and Giovannetti in 2007 \cite{physics} and independently by the AIM Minimum Rank - Special Graphs Workgroup in 2008 \cite{aim}. The original motivation for the latter came from the problem of bounding the minimum rank over all symmetric real matrices whose $ij$th entry ($i\neq j$) is nonzero whenever $ij$ is an edge of a graph $G$ and zero otherwise, while the former introduced this parameter to help them describe the controllability of certain quantum systems. Despite its beginnings in linear algebra and small applications in physics, the model has received considerable attention from combinatorialists due to its obvious ties to graph theory (\cite{what,today,plusone,conj2}).

The zero forcing process is a discrete-time process in which we start with a set $S$ of vertices of a graph $G$ which are initially colored black, while the remaining vertices are colored white. At each time step, the following rule is applied. Namely, if $u$ is a white vertex in $G$, it will be become black if it is the only white neighbor of some black vertex $v$, and we say that the vertex $v$ \textit{forced} $u$ to change color. 
 Once a vertex has been changed to black, it remains black forever. If every vertex of $G$ becomes black in finite time we say that $S$ is a \textit{zero forcing set}.
  Throughout this paper, every graph will be finite, simple, and undirected.

We define the zero forcing number of a graph $G$, denoted $Z(G)$, to be the mi\-ni\-mum cardinality over all \textit{zero forcing sets} of $G$.

Amos et al. \cite{1/3} proved that for a connected graph $G$ of order $n$ and maximum degree $\Delta\geq 2$
\[Z(G) \leq \frac{\Delta -2}{\Delta -1}n+\frac{2}{\Delta+1}.
\]
It is not difficult to show that this bound is attained exactly when $G$ is either $K_{\Delta +1}$, the complete bipartite graph $ K_{\Delta,\Delta}$ or a cycle. Later, pushing this bound a little further, Gentner and Rautenbach \cite{conj} were able to remove the additive constant $\frac{2}{\Delta+1}$ (for $\Delta \geq 3$). Namely, they showed that $Z(G) \leq \frac{\Delta - 2}{\Delta -1}n$ holds for every connected graph $G$ of order $n$ and maximum degree $\Delta \geq 3$, unless when $G$ is one of five exceptional graphs $K_{\Delta + 1}, K_{\Delta, \Delta}, K_{\Delta - 1, \Delta}$ or two other specific graphs (we do not exhibit them, for full details see \cite{conj}). Note that the zero forcing number of a connected graphs with maximum degree $2$ is completely understood. 
Indeed, for such graphs the forcing number is either $1$ in the case of a path or $2$ in the case of a cycle. However, even when the maximum degree is $3$, the following value
\[ z_3=\lim_{n \rightarrow \infty} \sup \{ \frac{Z(G)}{|V(G)|}: G \text{ connected, } |V(G)|\geq n \text{ and }\Delta(G)\leq 3\}
\] is not known. The currently best known upper bound for $z_3\leq 1/2$ is due to Amos et al. and follows from the result mentioned above. Furthermore, Gentner and Rautenbach (\cite{conj}), have proved that the upper bound of $n/2$ is far off when $G$ has maximum degree $3$ and girth at least $5$, where $n$ is the order of $G$. They showed that such graphs have zero forcing number at most $\frac{n}{2} - \frac{n}{24 \textrm{log}_2n +6} + 2$. We remark this result does not affect the best known upper bound for $z_3$ but suggests $1/2$ might not be the correct value. Motivated by this, the same authors conjectured that $Z(G)\leq \frac{1}{3}n+2$ for every connected graph $G$ with maximum degree $3$ \cite{conj}. 

In this short note, we disprove this conjecture by presenting an infinite family of connected graphs $\{G_n\}$, with maximum degree $3$, such that the zero forcing number of $G_n$ is at least $\frac{4}{9}|V(G_n)|$, thus proving $z_3\geq \frac{4}{9}$.

\section{Counterexamples to a conjecture of Gentner and Rautenbach}\label{sec:counter_example}
We create our counterexamples by substituting each leaf of a complete binary tree $B_d$ on on $2^d-1$ vertices ($d$ odd), by a complete graph on $4$ vertices with one of its edges subdivided (see Figure 1). 
Indeed, let $G_n$ ($n\geq 1$) be the graph obtained by replacing every leaf of $B_{2n-1}$ by the aforementioned subdivided $K_4$. We also denote $y_{n-1}^1$, $y_{n-1}^2$ to be the neighbors of $r_n$ in $G_n$ and $H_{n-1}^1$, $H_{n-1}^2$ to be the corresponding connected components of $G_n - r_n$. Observe that both subgraphs are isomorphic to the binary tree $B_{2n-2}$ with their leaves replaced by the subdivided $K_4$. Moreover, let $\widehat{G}_n$ be the graph obtained from $G_n$ by attaching a new leaf $y_n$ to the root $r_n$ of the underlying binary tree in $G_n$. Throughout this note, we will view $G_n$ as a subgraph of $\widehat{G}_n$ and containing $4$ induced copies of $G_{n-1}$.  Observe that the maximum degree of $G_n$ and $\widehat{G}_n$ is $3$, for all $n\geq 1$.  

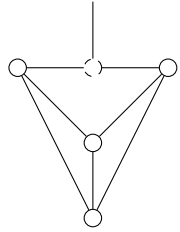
\begin{figure}[h]\label{K4}
\begin{tikzpicture}[every node/.style={circle, scale=0.7}]
\node (fake) at (7.5,2) [draw=white,dashed, fill=white, label=right:{}]{};
\node (g1)   at (15,2) [draw=white,dashed, fill=white, label=right:{}]{};
\node (a1)   at (14,1) [draw=black, fill=white]{};
\node (d1)   at (15,1) [draw=black,dashed, fill=white, label={[xshift=0.5cm,yshift=-0.3cm]}]{};
\node (e1)   at (15,0) [draw=black, fill=white]{};
\node (b1)   at (16,1) [draw=black, fill=white]{};
\node (c1)   at (15,-1) [draw=black, fill=white]{};

\draw (a1) -- (d1) -- (b1);
\draw (a1) -- (c1) -- (b1);
\draw (a1) -- (e1) -- (b1);
\draw (c1) -- (e1);
\draw (d1) -- (g1);
\end{tikzpicture}
 \caption{We substitute every leaf of $B_d$ by a subdivided $K_4$ (the dashed vertex denotes a leaf in $B_d$).}
\end{figure}

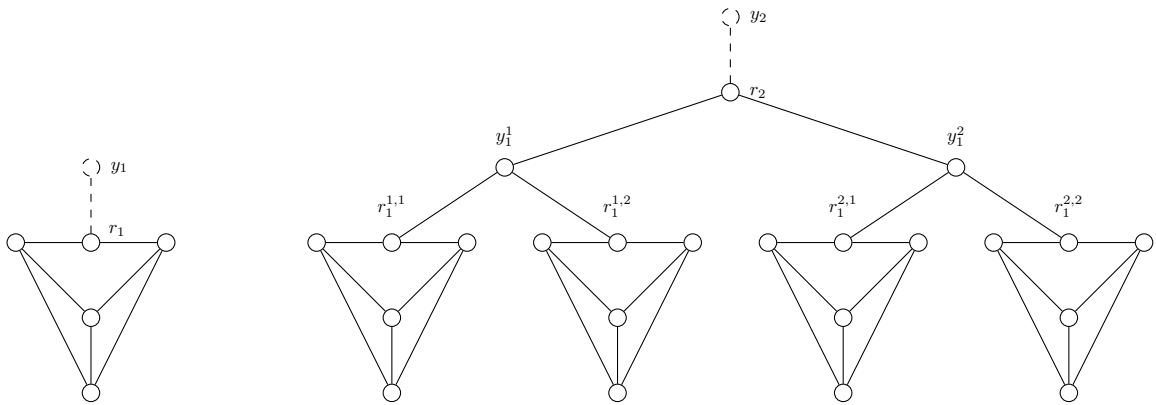
\begin{figure}[h]\label{fig:trees}
\begin{tikzpicture}[every node/.style={circle, scale=0.7}]

\node (g1)  at (0,2) [draw=black,dashed, fill=white, label=right:{$y_1$}]{};
\node (a1)  at (-1,1) [draw=black, fill=white]{};
\node (d1)  at (0,1) [draw=black, fill=white, label={[xshift=0.5cm,yshift=-0.3cm]$r_1$}]{};
\node (e1)  at (0,0) [draw=black, fill=white]{};
\node (b1)  at (1,1) [draw=black, fill=white]{};
\node (c1)  at (0,-1) [draw=black, fill=white]{};

\draw (a1) -- (d1) -- (b1);
\draw (a1) -- (c1) -- (b1);
\draw (a1) -- (e1) -- (b1);
\draw (c1) -- (e1);

\node (a11)  at (3,1) [draw=black, fill=white]{};
\node (d11)  at (4,1) [draw=black, fill=white,label = $r_1^{1,1}$]{};
\node (e11)  at (4,0) [draw=black, fill=white]{};
\node (b11)  at (5,1) [draw=black, fill=white]{};
\node (c11)  at (4,-1) [draw=black, fill=white]{};

\draw (a11) -- (d11) -- (b11);
\draw (a11) -- (c11) -- (b11);
\draw (a11) -- (e11) -- (b11);
\draw (c11) -- (e11);

\node (a12)  at (6,1) [draw=black, fill=white]{};
\node (d12)  at (7,1) [draw=black, fill=white,,label = $r_1^{1,2}$]{};
\node (e12)  at (7,0) [draw=black, fill=white]{};
\node (b12)  at (8,1) [draw=black, fill=white]{};
\node (c12)  at (7,-1) [draw=black, fill=white]{};

\draw (a12) -- (d12) -- (b12);
\draw (a12) -- (c12) -- (b12);
\draw (a12) -- (e12) -- (b12);
\draw (c12) -- (e12);

\node (h1)  at (5.5,2) [draw=black, fill=white, label = $y^1_1$]{};
\draw (d11) -- (h1) -- (d12);

\node (a13)  at (9,1) [draw=black, fill=white]{};
\node (d13)  at (10,1) [draw=black, fill=white,label = $r_1^{2,1}$]{};
\node (e13)  at (10,0) [draw=black, fill=white]{};
\node (b13)  at (11,1) [draw=black, fill=white]{};
\node (c13)  at (10,-1) [draw=black, fill=white]{};

\draw (a13) -- (d13) -- (b13);
\draw (a13) -- (c13) -- (b13);
\draw (a13) -- (e13) -- (b13);
\draw (c13) -- (e13);

\node (a14)  at (12,1) [draw=black, fill=white]{};
\node (d14)  at (13,1) [draw=black, fill=white,label = $r_1^{2,2}$]{};
\node (e14)  at (13,0) [draw=black, fill=white]{};
\node (b14)  at (14,1) [draw=black, fill=white]{};
\node (c14)  at (13,-1) [draw=black, fill=white]{};

\draw (a14) -- (d14) -- (b14);
\draw (a14) -- (c14) -- (b14);
\draw (a14) -- (e14) -- (b14);
\draw (c14) -- (e14);

\draw [dashed] (d1) -- (g1);

\node (h2)  at (11.5,2) [draw=black, fill=white, label = $y^2_1$]{};
\draw (d13) -- (h2) -- (d14);

\node (g2)  at (8.5,4) [draw=black,dashed, fill=white, label=right:{$y_2$}]{};
\node (x2)  at (8.5,3) [draw=black, fill=white, label=right:{$r_2$}]{};
\draw (h1) -- (x2) -- (h2);
\draw [dashed] (x2) -- (g2);
\end{tikzpicture}
 \caption{The graphs $G_1$ (left), and $G_2$ (right). The graphs $\widehat{G}_1$(left) and $\widehat{G}_2$(right) are the graphs containing $G_1$ and $G_2$ respectively, with the added dashed edge and vertex.}
\end{figure}

We take a closer look at the structure of $\widehat{G}_n$ to obtain the required lower bound on $Z(\widehat{G}_n)$. 
First, let the sequence $t_n$ be defined inductively as follows: $t_1 = 2$ and $t_{n+1} = 4t_n +2$ for every $n\geq 1$. 
Now we shall prove the following lemma.

\begin{lem}\label{binarylowerbound}
Let $F$ be a graph containing $\widehat{G}_n$ as an induced subgraph and such that there is no edge between $V(G_n)$ and $V(F)\backslash V(G_n)$. Then, for every zero forcing set $P$ of $F$, the following holds
\begin{itemize}
\item[i)] $|V(G_n) \cap P|\geq t_n$.
\item[ii)] If $|V(G_n) \cap P| = t_n$ then $r_n\not\in P$ and $V(G_n) \cap P$ does not force $r_n$ within $G_n$. 
\end{itemize}
\end{lem}
\begin{proof}
Both statements are straightforward for $n = 1$. For the inductive step, observe that if $|V(G_{n+1}) \cap P| \leq t_{n+1} - 1 = 4\cdot t_n + 1$, then we may assume $|V(H^1_n)\cap P|\leq 2\cdot t_n$. Since $|V(G^{1,1}_n)\cap P|$, $|V(G^{1,2}_n)\cap P|\geq t_n$ by induction, we must have $|V(G^{1,1}_n)\cap P| = |V(G^{1,2}_n)\cap P| = t_n$. From $(ii)$ we may deduce $r^{1,1}_n,r^{1,2}_n\not\in P$. Moreover, during the process none of these vertices can be forced by the vertices of $V(G^{1,1}_n)$ or $V(G^{1,2}_n)$, respectively. As a corollary,$r^{1,1}_n$ and $r^{1,2}$ must be forced by $y^1_n$, yet $y^1_n$ clearly can not force them simultaneously. This is a contradiction and it concludes the proof of part $i)$. Note that we have proved $|V(H^i_n) \cap P| \geq 2 \cdot t_n + 1$. 

Assume now that $|V(G_{n+1}) \cap P| = t_{n+1}$. Therefore, by the above, $|V(H^i_n) \cap P|=2 \cdot t_n + 1$ ($i\in \{1,2\}$), which implies $r_n\notin P$. Finally, suppose that $r_{n+1}\not\in P$ but it is forced during the process by a vertex in $G_{n+1}$. As $N_{G_{n+1}}(r_{n+1}) = \{y^1_n, y^2_n\}$, we may assume $y^1_n$ forced $r_{n+1}$. We proceed doing a casework:
\begin{itemize}
\item[a)] If $y^1_n\in P$, then we must have $|V(G^{1,1}_n) \cap P| = | V(G^{2,1}_n) \cap P| = t_n$. By the induction hypothesis, neither $r^{1,1}_n$ nor $r^{2,1}_n$ belong to $P$, and neither of them is forced by a vertex in their respective subgraph $G^{1,i}_n$. Thus $y^1_n$ cannot force $r_{n+1}$ as it has two unforced neighbors throughout the forcing process.
\item[b)] If $y^1_n\not\in P$, then it must be forced by $r^{1,1}_n$ or  $r^{1,2}_n$. Let us assume $r^{1,1}_n$ forced $y^1_n$, then we must have $|V(G^{1,1}_n) \cap P| \geq t_n + 1$ and we may deduce 
\begin{align*}
|V(G^{1,1}_n) \cap P|&=t_n +1 \text{ and }\\
|V(G^{1,2}_n) \cap P|&= t_n  \\
\end{align*}
Hence, again by induction, $r^{1,2}_n$ does not belong to $P$ and can not be forced within $G^{1,2}_n$. Although $y^1_n$ might indeed be forced by $r^{1,1}_n$, it still has two white neighbors $r^{1,2}_n$ and $r_{n+1}$ thus it cannot force $r_{n+1}$, which is a contradiction. This completes our case check and the proof of the lemma.
\end{itemize}
\end{proof}
\begin{corollary}
 $Z(\widehat{G}_n)\geq \frac{4}{9}|V(\widehat{G}_n)|$, for all $n\geq 1$. 
\end{corollary}
\begin{proof}
Observe that $t_n + 1 = \frac{8\cdot 4^{n-1}+1}{3}$ and $|V(\widehat{G}_n)|= 6\cdot 4^{n-1}$. Now, by Lemma \ref{binarylowerbound}, $Z(\widehat{G}_n)\geq t_n+1$ and therefore $Z(\widehat{G}_n) \geq \left( \frac{4}{9} + \frac{1}{18\cdot 4^{n-1}}\right)|V(\widehat{G}_n)|$. 
\end{proof}
We end this section by determining the exact value of the zero forcing numbers of $G_n$ and $\widehat{G}_n$.

\begin{proposition}\label{prop:tight}
$Z(G_n)=Z(\widehat{G}_n)= t_n+1$.
\end{proposition}
\begin{proof}
Lemma \ref{binarylowerbound} implies both $Z(G_n)$ and $Z(\widehat{G}_n)$ are greater or equal to $t_n+1$. We shall prove equality holds, by induction on $n$. To do so, we will prove a stronger assertion, namely that $G_n$ has a \textit{zero forcing set} $P_n$ of size $t_n+1$ satisfying the following properties:
\begin{itemize}
\item[a)] it contains  $r_n$,
\item[b)] $r_n$ does not need to force any of its neighbors.
\end{itemize}
The set $P_1$ can easily be found in $G_1$. For the inductive step, let $P_{n+1}$ be the union of $r_{n+1}$ with four ismorphic copies of the \textit{zero forcing set} $P_{n}$ inside each $G^{i,j}_n$ ($i,j \in \{1,2\}$), but with the two roots $r^{1,2}_n$ and $r^{2,2}_n$ removed. Clearly $P_{n+1}$ has size $4\cdot t_{n}+3=t_{n+1}+1$ and satifies $i)$. 
It is also easy to see that the vertices of both subgraphs $G^{1,1}_n$ and $G^{2,1}_n$ will be forced by the vertices in $P_{n+1}\cap G^{1,1}_n$ and $P_{n+1}\cap G^{1,1}_n$, respectively. (observe that this step requires the forcing to be completed without the active involvement of the root). Now, as $r_{n+1}$ is black, $y^1_n$ and $y^2_n$ will force $r^{1,2}_n$ and $r^{2,2}_n$, respectively. Using induction again it follows both $G^{1,2}_n$ and $G^{2,2}_n$ will become black. Hence, $P_{n+1}$ is a \textit{zero forcing set} and $r_{n+1}$ does not need to force any of its neighbours. From $ii)$ we deduce $P_{n+1}$ is also a \textit{zero forcing set} of $\widehat{G}_n$.
\end{proof}

\section{Additional Remarks}

One of the most interesting remaining questions in the field is to find the value of $z_3$. Knowing our constructions, we believe the result of Amos et al. gives the correct value of $z_3$. We formulate this belief as a conjecture:
\begin{conjecture}
$z_3=1/2$.
\end{conjecture}
The counterexamples we presented in this note used the idea of an appropriate "injection" of a subdivided $K_4$ in certain base graphs; we mention that, although the bound we obtained used binary trees as base graphs, we were able to beat the conjectured upper bound of $\frac{1}{3}n + 2$ using different base graphs. For example, we state the following result (without proof):
\begin{proposition}
Let $n$ be divisible by $6$ and let $C_n$ denote the cycle on $n$ vertices. Furthermore, set $\widehat{C}_n$ to be the graph obtained by attaching a distinct leaf to every vertex in $C_n$, and finally, let $G_n$ be constructed from $\widehat{C}_n$ by replacing every leaf with the subdivided $K_4$ graph. Then, $\frac{Z(G_n)}{|V(G_n)|}\geq \frac{5}{12}$.
\end{proposition}
It would be interesting to know if the presented injection technique with the appropriate choice of a base graph can imply even better lower bounds on $z_3$.

\section{Acknowledgment} We would like to thank the anonymous reviewers for their careful reading of our manuscript and their many insightful comments and suggestions that improved the presentation of our article.

The first author would like to thank B\'ela Bollob\'as for the invitation to visit the University of Memphis while this research has been done.

\end{document}